\documentclass[12pt]{amsart}
\usepackage{amsmath,amsfonts,amstext,amsthm,epsfig}

\newtheorem{Theorem}{Theorem}[section]
\newtheorem{Lemma}[Theorem]{Lemma}
\newtheorem{Corol}[Theorem]{Corollary}
\newtheorem{Prop}[Theorem]{Proposition}
\newtheorem{Rem}[Theorem]{Remark}
\newtheorem{Def}[Theorem]{Definition}
\newtheorem{Hyp}[Theorem]{Hypothesis}
\makeatletter
    \addtocounter{section}{0}
    
     \@addtoreset{equation}{section}
\makeatother

\def\N{\mathbb N}
\def\R{\mathbb R}

\begin{document}

\title[Regularity of invariant sets]
{Regularity of invariant sets in semilinear damped wave equations}%
\author{Martino Prizzi}

\address{Martino Prizzi, Universit\`a di Trieste, Dipartimento di
Matematica e Informatica, Via Valerio 12/1, 34127 Trieste, Italy}%
\email{prizzi@dmi.units.it}%

\date{\today}%
\begin{abstract} Under fairly general assumptions, we prove that every compact invariant subset $\mathcal I$ of the semiflow generated
by the semilinear damped wave equation
\begin{equation*}
\begin{aligned}
\epsilon
u_{tt}+u_t+\beta(x)u-\sum_{ij}(a_{ij}(x)u_{x_j})_{x_i}&=f(x,u),&&(t,x)\in[0,+\infty[\times\Omega,\\
u&=0,&&(t,x)\in[0,+\infty[\times\partial\Omega
\end{aligned}\end{equation*}
in $H^1_0(\Omega)\times L^2(\Omega)$ is in fact bounded in $D(\mathbf A)\times H^1_0(\Omega)$.
Here $\Omega$ is an arbitrary, possibly unbounded, domain in $\R^3$, $\mathbf A u=\beta(x)u-\sum_{ij}(a_{ij}(x)u_{x_j})_{x_i}$
is a positive selfadjoint elliptic operator and $f(x,u)$ is a nonlinearity of critical growth. The nonlinearity $f(x,u)$ needs not to satisfy any dissipativeness assumption and the
invariant subset $\mathcal I$ needs not to be an an attractor.
\end{abstract}
\maketitle

\section{Introduction}
Consider the semilinear damped wave equation
\begin{equation}\label{intro1}
\begin{aligned}
\epsilon
u_{tt}+u_t+\beta(x)u-\sum_{ij}(a_{ij}(x)u_{x_j})_{x_i}&=f(x,u),&&(t,x)\in[0,+\infty[\times\Omega,\\
u&=0,&&(t,x)\in[0,+\infty[\times\partial\Omega,
\end{aligned}\end{equation}
where $\Omega$ is an arbitrary, possibly unbounded, domain in
$\R^3$, $f(x,u)$ is a nonlinearity of critical growth and $\mathbf
A u:=\beta(x)u-\sum_{ij}(a_{ij}(x)u_{x_j})_{x_i}$ is a positive
selfajoint elliptic operator. It is well known (see e-g.
\cite{Tem}) that equation (\ref{intro1}), under appropriate
conditions on $a_{ij}(x)$, $\beta(x)$ and $f(x,u)$, generates a
(local) semiflow in the space $H^1_0(\Omega)\times L^2(\Omega)$.
We remind that a subset $\mathcal S$ of $H^1_0(\Omega)\times
L^2(\Omega)$ is called {\em invariant} for the semiflow generated
by (\ref{intro1}) if for every $(u_0,v_0)\in\mathcal S$ there
exists a solution $(u(\cdot),v(\cdot))\colon \R\to
H^1_0(\Omega)\times L^2(\Omega)$ of (\ref{intro1}) with
$(u(0),v(0))=(u_0,v_0)$ and $(u(t),v(t))\in\mathcal S$ for all
$t\in\R$. Assume that $\mathcal I$ is a compact invariant subset
for this semiflow. In this paper we shall prove that, under fairly
general assumptions on $a_{ij}(x)$, $\beta(x)$ and $f(x,u)$,
$\mathcal I$ is in fact bounded in $D(\mathbf A)\times
H^1_0(\Omega)$. This means that a solution of (\ref{intro1}) lying
in $\mathcal I$ is more regular in space than a generic solution.
Results of this kind have been known for a long time in case
$f(x,u)$ satisfies some dissipativeness condition and
 $\mathcal I$ is the global attractor of (\ref{intro1}) (see e.g. \cite{BV,Har, HR1,GhiTem,EM} and the more
 recent \cite{GraPa1,GraPa2,PaZe,CoPa}). In \cite{Ryba} regularity
 results were obtained for general invariant subsets in the subcritical case.
  To our knowledge, the most general results are contained in the
 paper \cite{HR2} by Hale and Raugel,
 where the authors, among other things, prove ``spatial regularity" of invariant subsets for a general
class of abstract semilinear evolution equations.
 The equations considered by Hale and Raugel are of the form $\dot u=Au+f(u)$, where $A$ is the generator of
 a $C^0$-semigroup of linear operators in a Banach space $X$ and $f$ is a
 nonlinearity of class $C^{1,1}$. The assumptions in \cite{HR2} are too elaborated to be summarized here.
 The technique relies on suitable Galerkin decompositions of the solutions lying in the invariant subset.
 Roughly speaking, every solution $u(t)$ in the invariant subset
splits as $u(t)=v(t)+w(t)$, where $w$ is the  fixed point of an
integral equation and $v(t)$ is the solution of a retarded
differential equation on a (usually finite dimensional) subspace
of $X$. The applications described in \cite{HR2} consider only the
case of equations on bounded domains, where a natural Galerkin
decomposition is supplied by the (finite dimensional) spectral
projections. However,  it is very likely that the abstract results
of \cite{HR2} should apply also to the case of equations on
unbounded domains. In this case, the decomposition on a basis of
eigenfunctions should be replaced by the use of the spectral
family of the operator $A$.

 Our aim is to go beyond the results of \cite{HR2} in the particular case of the semilinear damped wave equation (\ref{intro1}).
 We shall prove our regularity results without any smoothness and/or boundedness assumption on $\Omega$.
 The nonlinearity $f(x,u)$ needs not to be of class $C^{1,1}$ in $u$, but only of class $C^{1,\beta}$ for some $0<\beta<1$.
 Moreover, we shall not exploit Galerkin decompositions of the solutions, so we bypass the problem of constructing spectral families.
 Finally, we do not need to use the theory of retarded differential equations.

 The idea of the proof is very simple, although it requires a careful functional analytic setting. We give here an informal
 sketch. Let $(\bar u(\cdot),\bar u_t(\cdot))\colon\R\to H^1_0(\Omega)\times L^2(\Omega)$ be a
 bounded mild solution of (\ref{intro1}). Set $\bar v(t):=\bar u_t(t)$. Then
 $(\bar v(\cdot),\bar v_t(\cdot))\colon \R\to L^2(\Omega)\times H^{-1}(\Omega)$ is a mild
 solution of
\begin{equation}\label{intro2}
\begin{aligned}
\epsilon
v_{tt}+v_t+\beta(x)v-\sum_{ij}(a_{ij}(x)v_{x_j})_{x_i}&=\partial_u
f(x,\bar u(t))v,&&(t,x)\in[0,+\infty[\times\Omega,\\
v&=0,&&(t,x)\in[0,+\infty[\times\partial\Omega.
\end{aligned}
\end{equation}
Take $\theta>0$ and denote by $\mathbf U(t,s)$ the evolution
system generated by the non-autonomous linear equation
\begin{equation}\label{intro3}
\begin{aligned}
\epsilon
v_{tt}+v_t+\beta(x)v-\sum_{ij}(a_{ij}(x)v_{x_j})_{x_i}+\theta v
-\partial_u f(x,\bar
u(t))v&=0,&&(t,x)\in[0,+\infty[\times\Omega,\\
v&=0,&&(t,x)\in[0,+\infty[\times\partial\Omega
\end{aligned}
\end{equation}
in the space $L^2(\Omega)\times H^{-1}(\Omega)$. Then, for $t\geq
s$, we have that
\begin{equation}\label{intro4}
(\bar v(t),\bar v_t(t))=\mathbf U(t,s)(\bar v(s),\bar
v_t(s))+\int_s^t\mathbf U(t,p)(0,(\theta/\epsilon)\bar v(p))\,dp.
\end{equation}
 We shall prove in Theorem \ref{decay} below that, if $\theta$ is sufficiently large, then $\mathbf U(t,s)$ satisfies
appropriate exponential decay estimates in $L^2(\Omega)\times
H^{-1}(\Omega)$ as well as in $H^1_0(\Omega)\times L^2(\Omega)$.
Then, letting $s\to-\infty$ in (\ref{intro4}), we obtain that
\begin{equation}\label{intro5}
(\bar v(t),\bar v_t(t))=\int_{-\infty}^t\mathbf
U(t,p)(0,(\theta/\epsilon)\bar v(p))\,dp.
\end{equation}
In this way we get rid of the Cauchy data $(\bar v(s),\bar
v_t(s))$ and, since $(0,(\theta/\epsilon)\bar v(p))\in
H^1_0(\Omega)\times L^2(\Omega)$, we deduce that actually $(\bar
v(\cdot),\bar v_t(\cdot))$ is  a bounded function from $\R$ into
$H^1_0(\Omega)\times L^2(\Omega)$ and the conclusion follows. A similar idea was already exploited in \cite{EM}.

The paper is organized as follows. In Section 2 we introduce notations, we set the main assumptions and we state
the main results. Sections 3 and 4 are devoted to the proof of the main results. In Section 5 we exploit the regularity
results to prove  upper-semicontinuity of the attractors of (\ref{intro1}) as $\epsilon\to 0$ when $f(x,u)$ is dissipative,
improving a previous result obtained with K. Rybakowski \cite{PR3}.

\section {Notation, statements and remarks}

Before we describe in detail our assumptions and our results, we
need to introduce some notation. In this paper $\Omega$
is an arbitrary open subset of $\R^3$, bounded or not.   Given  a function $g\colon \Omega\times \R\to \R$,
we denote by $\hat g$  the Nemitski operator which
associates with every function $u\colon \Omega\to \R$ the function
$\hat g(u)\colon \Omega\to \R$ defined by
 $$
  \hat g(u)(x)= g(x,u(x)),\quad x\in \Omega.
 $$

If $I\subset \R$, $Y$ and $X$ are normed spaces with $Y\subset X$
and if $u\colon I\to Y$ is a function which is differentiable as a
function into $X$ then we denote its $X$-valued derivative by
$(\partial_t \mid X)\,u$. Similarly, if $X$ is a Banach space and
$u\colon I\to X$ is integrable as a function into $X$, then we
denote its $X$-valued integral by $\int _I u(t)\,(dt\mid X)$. If
$X$ and $Y$ are Banach spaces, we denote by $\mathcal L(X,Y)$ the
space of bounded linear operators from $X$ to $Y$. If $X=Y$ we
write just $\mathcal L(X)$.

\begin{Hyp}\label{hyp1}\
\begin{enumerate}
\item$a_0$,
$a_1\in]0,\infty[$ are constants and $a_{ij}\colon
\Omega\to \R$ are functions in
$L^\infty(\Omega)$ such that $a_{ij}=a_{ji}$, $i$,
$j=1$, \dots, $3$, and for every $\xi\in\R^3$ and a.e.
$x\in\Omega$,
$$a_0|\xi|^2\le \sum_{i,j=1}^3
a_{ij}(x)\xi_i\xi_j\le a_1|\xi|^2 .$$
\item $\beta\colon
\Omega\to \R$ is a measurable function with the property
that
\begin{enumerate}
\item for every $\nu>0$ there is a $C_\nu>0$ with
$$\int_\Omega|\beta(x)||u(x)|^2\,dx\leq \nu\int_\Omega|\nabla u(x)|^2\,dx+C_\nu \int_\Omega|u(x)|^2\,dx
$$ for all $u\in
H^1_0(\Omega)$;
\item there exists $\lambda_1>0$ such that, setting  $A(x):=(a_{ij}(x))_{i,j=1}^3$,
$$\int_\Omega A(x)\nabla u(x)\cdot\nabla u(x)\,dx+\int_\Omega\beta(x)|u(x)|^2\,dx\geq\lambda_1\int_\Omega |u(x)|^2\,dx$$ for all $u\in
H^1_0(\Omega)$.\end{enumerate}\end{enumerate}
\end{Hyp}
\begin{Rem}Condition {\em(a)} in Hypothesis \ref{hyp1} is satisfied, e.g., if
$\beta\in L^p_{\mathrm u}(\R^3)$ with $p>3/2$. Here we denote by
$L^p_{\mathrm u}(\R^3)$ the set of  measurable functions
$\zeta\colon \R^3\to \R$ such that $$
 \|\zeta\|_{L^p_{\mathrm u}}:=\sup_{y\in \R^3}\left(\int_{
 B(y)}|\zeta(x)|^p\,d x\right)^{1/p}<\infty,
$$ where, for $y\in\R^3$, $B(y)$ is the open unit cube in $\R^3$
centered at $y$ (see \cite{PR2} for details).
\end{Rem}

By Lemma 3.4 in \cite{PR2}, the scalar product
\begin{equation}
\langle u,v\rangle_{H^1_0}=\int_\Omega A(x)\nabla u(x)\cdot\nabla
v(x)\,dx+\int_\Omega\beta(x)u(x)v(x)\,dx,\quad u,v\in H^1(\Omega)
\end{equation}
is equivalent to the usual scalar product on $H^1_0(\Omega)$. From
now on,  we denote by $\|\cdot\|_{H^1_0}$ the norm associated with
$\langle \cdot,\cdot\rangle_{H^1_0}$.

Let $\mathbf A$ be the selfadjoint operator on $L^2(\Omega)$
defined by the differential operator $u\mapsto\beta
u-\sum_{ij}(a_{ij}u_{x_j})_{x_i}$. Then $\mathbf A$ generates a
family $X^\kappa$, $\kappa\in\R$, of fractional power spaces with
$X^{-\kappa}$ being the dual of $X^\kappa$ for
$\kappa\in]0,+\infty[$. For $\kappa\in]0,+\infty[$, the space
$X^\kappa$ is a Hilbert space with respect to the scalar product
\begin{equation*}
\langle u,v\rangle_{X^\kappa}:=\langle{\mathbf A}^\kappa
u,{\mathbf A}^\kappa v\rangle_{L^2},\quad u,v\in X^\kappa.
\end{equation*}
Also, the space $X^{-\kappa}$ is a Hilbert space with respect to
the scalar product $\langle \cdot,\cdot\rangle_{X^{-\kappa}}$ dual
to the scalar product  $\langle \cdot,\cdot\rangle_{X^{\kappa}}$,
i.e.
\begin{equation*}
\langle u',v'\rangle_{X^{-\kappa}}=\langle R^{-1}_\kappa
u',R^{-1}_\kappa v'\rangle_{X^\kappa},\quad u,v\in X^{-\kappa},
\end{equation*}
where $R_\kappa\colon X^\kappa\to X^{-\kappa}$ is the Riesz
isomorphism $u\mapsto\langle\cdot,u\rangle_{X^\kappa}$.

 We write
\begin{equation*}
H_\kappa=X^{\kappa/2},\quad\kappa\in\R.
\end{equation*}
Note that $H_0=L^2(\Omega)$, $H_1=H^1_0(\Omega)$,
$H_{-1}=H^{-1}(\Omega)$ and $H_2=D({\mathbf A})$.

For $\kappa\in\R$ the operator $\mathbf A$ induces a selfadjoint
operator $\mathbf A_\kappa\colon H_{\kappa+2}\to H_{\kappa}$. In
particular $\mathbf A=\mathbf A_0$. Moreover,
\begin{equation*}
\langle u,v\rangle_{H^1_0}=\langle\mathbf A_0
u,v\rangle_{L^2},\quad u\in D(\mathbf A_0),\,v\in H^1_0(\Omega).
\end{equation*}

For $\epsilon\in]0,1]$ and $\kappa\in\R$ set
$Z_\kappa:=H_{\kappa+1}\times H_{\kappa}$ and define the linear
operator $\mathbf B_{\epsilon,\kappa}\colon Z_{\kappa+1}\to
Z_{\kappa}$ by
\begin{equation*}
\mathbf B_{\epsilon,\kappa}(u,v):=(v,-(1/\epsilon)(v+\mathbf
A_\kappa u)),\quad (u,v)\in Z_{\kappa+1}.
\end{equation*}
It follows that $\mathbf B_{\epsilon,\kappa}$ is $m$-dissipative
on $Z_{\kappa}$ (cf the proof of Prop. 3.6 in \cite{PR2}).
Therefore, by the Hille-Yosida-Phillips theorem (see e.g.
\cite{CH}), $\mathbf B_{\epsilon,\kappa}$ is the infinitesimal
generator of a $C^0$-semigroup $\mathbf T_{\epsilon,\kappa}(t)$,
$t\in[0,+\infty[$, on $Z_{\kappa}$.

\begin{Hyp}\label{hyp2}\
\begin{enumerate}
\item $f\colon\Omega\times\R\to\R$ is such that, for every $u\in\R$, $f(\cdot,u)$ is measurable and $f(\cdot,0)\in L^2(\Omega)$;
\item for a.e. $x\in\Omega$, $f(x,\cdot)$ is of class $C^1$, $\partial_uf(\cdot,0)\in L^\infty(\Omega)$ and there exist constants $C$, $\beta$ and $\alpha$, with $C>0$, $0<\beta\leq 1$, $1\leq\alpha<2$
and $\alpha+\beta=2$, such that
$$
|\partial_u f(x,u_1)-\partial_u f(x,u_2)|\leq C(1+|u_1|^\alpha+|u_2|^\alpha)|u_1-u_2|^\beta.
$$
\end{enumerate}\end{Hyp}

The main properties of the Nemitski operator associated with $f$
are collected in the following Proposition, whose proof is left to
the reader.
\begin{Prop}\label{prop1} Assume Hypothesis \ref{hyp2}. Then
$\hat f\colon H^1_0(\Omega)\to L^2(\Omega)$ is continuously
differentiable, $D\hat f(u)[v](x)=\partial_u f(x,u(x))v(x)$ for
$u$, $v\in H^1_0(\Omega)$, and there exists a positive constant
$\tilde C>0$ such  that the following estimates hold:
\begin{equation}\label{n1}
\|\hat f(u)\|_{L^2}\leq\tilde C(1+\|u\|_{H^1_0}^3),\quad u\in
H^1_0(\Omega)
\end{equation}
\begin{equation}\label{n2}
\|D\hat f (u)\|_{{\mathcal L}(H^1_0,L^2)}\leq \tilde
C(1+\|u\|_{H^1_0}^2),\quad u\in H^1_0(\Omega)
\end{equation}
\begin{multline}\label{n3}
\|D\hat f (u_1)-D\hat f (u_2)\|_{{\mathcal L}(H^1_0,L^2)}\leq
\tilde
C(1+\|u_1\|_{H^1_0}^\alpha+\|u_\alpha\|_{H^1_0}^\alpha)\|u_1-u_2\|_{H^1_0}^\beta,\\
u_1,u_2\in H^1_0(\Omega).
\end{multline}
If $u\in H^1_0(\Omega)$ and $v\in L^2(\Omega)$, then
$\widehat{\partial_u f}(u)\cdot v\in H^{-1}(\Omega)$ and the
following estimates hold:
\begin{equation}\label{n4}
\|\widehat{\partial_u f}(u)\|_{\mathcal L(L^2,H^{-1})}\leq\tilde
C(1+\|u\|_{H^1_0}^2),\quad u\in H^1_0(\Omega)
\end{equation}
\begin{multline}\label{n5}
\|\widehat{\partial_u f}(u_1)-\widehat{\partial_u
f}(u_2)\|_{\mathcal L(L^2,H^{-1})}\leq\tilde
C(1+\|u_1\|_{H^1_0}^\alpha+\|u_2\|_{H^1_0}^\alpha)\|u_1-u_2\|_{H^1_0}^\beta,\\
u_1,u_2\in H^1_0(\Omega).
\end{multline}
Finally, whenever the function $t\mapsto u(t)$ is continuous from
$\R$ to $H^1_0(\Omega) $ and continuously differentiable from $\R$
to $L^2(\Omega)$, then the function $t\mapsto\hat f(u(t))$ is
continuously differentiable from $\R$ to $H^{-1}(\Omega)$ and
\begin{equation}\label{n6}
(\partial_t\mid H^{-1})(\hat f \circ u )(t)=\widehat{\partial_u
f}(u(t))\cdot(\partial_t\mid L^2)u(t)
\end{equation}
\qed
\end{Prop}

We consider the following semilinear damped wave equation:
\begin{equation}\label{eq1}
\begin{aligned}
\epsilon
u_{tt}+u_t+\beta(x)u-\sum_{ij}(a_{ij}(x)u_{x_j})_{x_i}&=f(x,u),&&(t,x)\in[0,+\infty[\times\Omega,\\
u&=0,&&(t,x)\in[0,+\infty[\times\partial\Omega
\end{aligned}\end{equation}
with Cauchy data $u(0)=u_0$, $u_t(0)=v_0$.
\begin{Rem}
The condition $\lambda_1>0$ in Hypothesis \ref{hyp1} is not
restrictive. Indeed, if Hypothesis \ref{hyp1} is satisfied with
$\lambda_1\leq 0$, one can take some $\gamma>0$ and add
$-\lambda_1u+\gamma u$ on both sides of (\ref{eq1}); then
Hypotheses \ref{hyp1} and \ref{hyp2} are fully satisfied, with
$\beta(x)$ replaced by $\beta(x)-\lambda_1+\gamma$ and $f(x,u)$
replaced by $f(x,u)-\lambda_1u+\gamma u$.
\end{Rem}

Following \cite{CH}, we rewrite equation (\ref{eq1}) as an
integral evolution equation in the space $Z_0=H^1_0(\Omega)\times
L^2(\Omega)$, namely
\begin{equation}\label{eq2}
(u(t),v(t))=\mathbf T_{\epsilon,0}(t)(u_0,v_0)+\int_0^t \mathbf
T_{\epsilon,0}(t-p)(0,(1/\epsilon)\hat f(u(p)))\,(dp\mid Z_0).
\end{equation}
Equation (\ref{eq2}) is called the {\em mild formulation} of (\ref{eq1}) and  solutions of (\ref{eq2})
are called {\em mild solutions} of (\ref{eq1}).
Note that by Proposition \ref{prop1} the nonlinear operator
$(u,v)\mapsto(0,\hat f(u))$ is Lipschitz continuos from $Z_0$ into
itself. Therefore,  if $(u_0,v_0)\in Z_0$, then (\ref{eq2})
possesses a unique continuous maximal solution
$(u(\cdot),v(\cdot))\colon[0,t_{\mathrm{max}}[\to Z_0$ (see Theor.
4.3.4 and Prop. 4.3.7 in \cite{CH}). We thus obtain a local
semiflow on $Z_0$. Moreover, $(u(\cdot),v(\cdot))$ is continuously
differentiable into $Z_{-1}$ and
\begin{equation}\label{eq3}
(\partial_t\mid Z_{-1})(u(t),v(t))=\mathbf
B_{\epsilon,-1}(u(t),v(t))+(0,(1/\epsilon)\hat f(u(t)))
\end{equation}
(see Theorem II.1.3 in \cite{Gold}). In particular, one has
\begin{equation}\label{eq4}\begin{cases}
(\partial_t\mid H_0)u(t)=v(t)\\ \epsilon(\partial_t\mid
H_{-1})v(t)=-v(t)-\mathbf A_{-1}u(t)+\hat f(u(t))
\end{cases}
\end{equation}
\begin{Def}\label{def1}
A function $(u(\cdot), v(\cdot))\colon\R\to Z_0$ is called a {\em
full solution} of (\ref{eq2}) iff, for every $s$, $t\in\R$, with
$s\leq t$, one has
\begin{equation*}
(u(t),v(t))=\mathbf T_{\epsilon,0}(t-s)(u(s),v(s))+\int_s^t
\mathbf T_{\epsilon,0}(t-p)(0,(1/\epsilon)\hat f(u(p)))\,(dp\mid
Z_0).
\end{equation*}
\end{Def}

Now we can state our first main result.
\begin{Theorem}\label{th1}
Assume that Hypotheses \ref{hyp1} and \ref{hyp2} are satisfied.
Let $\epsilon\in]0,1]$ be fixed. Let $(\bar u(\cdot), \bar
v(\cdot))\colon\R\to Z_0$ be a bounded full solution of
(\ref{eq2}), such that $\sup_{t\in\R}(\|\bar
u(t)\|_{H^1_0}^2+\epsilon\|\bar v(t)\|_{L^2}^2)\leq R$. Assume
that the first component $\bar u(\cdot)$ is uniformly continuous
with modulus of continuity $\omega(\cdot)$. Then $(\bar u(\cdot),\bar
v(\cdot))$ is continuous into $Z_1$, is continuously
differentiable into $Z_0$, and
\begin{equation*}
(\partial_t\mid Z_{0})(\bar u(t),\bar v(t))=\mathbf
B_{\epsilon,0}(\bar u(t),\bar v(t))+(0,(1/\epsilon)\hat f(\bar
u(t))).
\end{equation*}
Moreover, there exists a positive constant $\tilde R_\epsilon$ such that
\begin{equation*}
\sup_{t\in\R}(\|\mathbf A_0 \bar u(t)\|_{L^2}^2+\|\bar
v(t)\|_{H^1_0}^2+\epsilon\|(\partial_t\mid H_0)\bar
v(t)\|_{L^2}^2)\leq \tilde R_\epsilon.
\end{equation*}
The constant $\tilde R_\epsilon$ depends, besides $\epsilon$, only on the
constants in Hypotheses \ref{hyp1} and \ref{hyp2}, on $R$ and on
$\omega(\cdot)$.
\end{Theorem}

We remind that a subset $\mathcal I$ of $Z_0$ is called {\em
invariant} for the semiflow generated by (\ref{eq2}) if for every
$(u_0,v_0)\in\mathcal I$ there exists a full solution
$(u(\cdot),v(\cdot))$ of (\ref{eq2}) with $(u(0),v(0))=(u_0,v_0)$
and $(u(t),v(t))\in\mathcal I$ for all $t\in\R$.

\begin{Lemma}[{\bf Lemma 2.3 in \cite{HR2}}]\label{modcont}
If $\mathcal I$ is a compact invariant subset for the semiflow generated by (\ref{eq2}), then the set of all
the full solutions of (\ref{eq2}) in $\mathcal I$ is uniformly equicontinuous. \qed\end{Lemma}

Therefore, if $\mathcal I$ is a compact invariant subset for the semiflow generated by (\ref{eq2}), then there exists a continuous, non decreasing function
$\omega\colon[0,1]\to\R_+$, with $\omega(0)=0$, such that, for every full solution $(u(\cdot),v(\cdot))$ of  (\ref{eq2}) in $\mathcal I$, one has
\begin{equation*}
\|u(t)-u(s)\|_{H^1_0}\leq\omega(|t-s|),\quad t,s\in\R,\quad |t-s|\leq 1.
\end{equation*}

 As a consequence of Theorem \ref{th1} and Lemma \ref{modcont}, one can easily prove the following
corollary.

\begin{Corol}\label{cor1}
Assume that Hypotheses \ref{hyp1} and \ref{hyp2} are satisfied.
Let $\epsilon\in]0,1]$ be fixed. Let $\mathcal I$ be a compact
invariant subset of the local semiflow generated by (\ref{eq2}) in
$Z_0$. Then $\mathcal I$ is a bounded subset of $Z_1$.\qed
\end{Corol}

Theorem \ref{th1} and Corollary \ref{cor1} furnish estimates which
depend heavily on $\epsilon$. In many situations it is of interest
to obtain estimates which are uniform in $\epsilon$. To this end,
we need to introduce the following hypothesis.
\begin{Hyp}\label{hyp3}\
\begin{enumerate}
\item $f\colon\Omega\times\R\to\R$ is such that, for every $u\in\R$, $f(\cdot,u)$ is measurable and $f(\cdot,0)\in L^2(\Omega)$;
\item for a.e. $x\in\Omega$, $f(x,\cdot)$ is of class $C^2$, $\partial_uf(\cdot,0)\in L^\infty(\Omega)$, $\partial_{uu}f(\cdot,0)\in L^\infty(\Omega)$
and there exists a constants $C>0$ such that $$ |\partial_{uu}
f(x,u_1)-\partial_{uu} f(x,u_2)|\leq C|u_1-u_2|. $$
\end{enumerate}\end{Hyp}
Notice that Hypothesis \ref{hyp3} is a strenghtening of Hypothesis
\ref{hyp2}. We have the following theorem.

\begin{Theorem}\label{th2} Assume that Hypotheses \ref{hyp1} and \ref{hyp3} are
satisfied. For every $\epsilon\in]0,1]$, let $(\bar
u_\epsilon(\cdot), \bar v_\epsilon(\cdot))\colon\R\to Z_0$ be a
bounded full solution of (\ref{eq2}), such that
$\sup_{t\in\R}(\|\bar u_\epsilon(t)\|_{H^1_0}^2+\epsilon\|\bar
v_\epsilon(t)\|_{L^2}^2)\leq R$. Assume that, for every
$\epsilon\in]0,1]$, the first component $\bar u_\epsilon(\cdot)$
is uniformly continuous. Then  there exists a positive constant
$\tilde R$ such that, for every $\epsilon\in]0,1]$,
\begin{equation*}
\sup_{t\in\R}(\|\mathbf A_0 \bar u_\epsilon(t)\|_{L^2}^2+\|\bar
v_\epsilon(t)\|_{H^1_0}^2+\epsilon\|(\partial_t\mid H_0)\bar
v_\epsilon(t)\|_{L^2}^2)\leq \tilde R.
\end{equation*}
The constant $\tilde R$ depends only on the constants in
Hypotheses \ref{hyp1} and \ref{hyp3} and on $R$.
\end{Theorem}

One has also the following corollary.

\begin{Corol}\label{cor2}
Assume that Hypotheses \ref{hyp1} and \ref{hyp3} are satisfied.
For every $\epsilon\in]0,1]$, let $\mathcal I_\epsilon$ be a
compact invariant subset of the local semiflow generated by
(\ref{eq2})
 in $Z_0$. Assume that there exists $R>0$ such that,
for every $\epsilon\in]0,1]$,
\begin{equation*}
\sup_{(u,v)\in\mathcal I_\epsilon}(\| u\|_{H^1_0}^2+\epsilon\|
v\|_{L^2}^2)\leq R.
\end{equation*}
Then there exists $\tilde R>0$ such that, for every
$\epsilon\in]0,1]$,
\begin{equation*}
\sup_{(u,v)\in\mathcal I_\epsilon}(\|\mathbf A_0 u\|_{L^2}^2+\|
v\|_{H^1_0}^2)\leq \tilde R.
\end{equation*}
The constant $\tilde R$ depends only on the constants in
Hypotheses \ref{hyp1} and \ref{hyp3} and on $R$.\qed
\end{Corol}

\section{Proof of Theorem 1}

Throughout this section we fix $\epsilon\in]0,1]$ and we denote by $(\bar u(\cdot), \bar v(\cdot))\colon\R\to Z_0$  a fixed bounded
full solution of (\ref{eq2}), such that $\sup_{t\in\R}(\|\bar
u(t)\|_{H^1_0}^2+\epsilon\|\bar v(t)\|_{L^2}^2)\leq R$. As we have
seen above, $(\bar u(\cdot),\bar v(\cdot))$ is continuously
differentiable into $Z_{-1}$ and
\begin{equation*}\begin{cases}
(\partial_t\mid H_0)\bar u(t)=\bar v(t)\\ \epsilon(\partial_t\mid
H_{-1})\bar v(t)=-\bar v(t)-\mathbf A_{-1}\bar u(t)+\hat f(\bar
u(t))
\end{cases}
\end{equation*}
Set $\bar w(t):=(\partial_t\mid H_{-1})\bar v(t)$, $t\in\R$. Using
(\ref{n6}) we see that $(\bar v(\cdot),\bar w(\cdot))$ is
continuous into $Z_{-1}$ and continuously differentiable into
$Z_{-2}$, and
\begin{equation*}\begin{cases}
(\partial_t\mid H_{-1})\bar v(t)=\bar w(t)\\
\epsilon(\partial_t\mid H_{-2})\bar w(t)=-\bar w(t)-\mathbf
A_{-2}\bar v(t)+\widehat{\partial_u f}(\bar u(t))\cdot\bar v(t)
\end{cases}
\end{equation*}
Since the mapping $t\mapsto(0, \widehat{\partial_u f}(\bar
u(t))\cdot\bar v(t))$ is continuous into $Z_{-1}=D(\mathbf
B_{\epsilon,-2})$, it follows from Theorem II.1.3 in \cite{Gold}
that, for $s$, $t\in\R$, with $s\leq t$, $(\bar v(\cdot),\bar
w(\cdot))$ satisfies the equality
\begin{multline*}
(\bar v(t),\bar w(t))=\mathbf T_{\epsilon,-2}(t-s)(\bar v (s),\bar
w (s))\\+\int_s^t \mathbf
T_{\epsilon,-2}(t-p)(0,(1/\epsilon)\widehat{\partial_u f}(\bar
u(p))\cdot\bar v(p))\,(dp\mid Z_{-2}).
\end{multline*}
Finally, since $(\bar v(\cdot),\bar w(\cdot))$ is continuous into
$Z_{-1}$, it follows  that $(\bar v(\cdot),\bar w(\cdot))$
satisfies the equality
\begin{multline}\label{k1}
(\bar v(t),\bar w(t))=\mathbf T_{\epsilon,-1}(t-s)(\bar v (s),\bar
w (s))\\+\int_s^t \mathbf
T_{\epsilon,-1}(t-p)(0,(1/\epsilon)\widehat{\partial_u f}(\bar
u(p))\cdot\bar v(p))\,(dp\mid Z_{-1}).
\end{multline}
Notice that, for $s\in\R$ fixed, the function $(\bar v(\cdot),\bar
w(\cdot))$ is the unique mild solution of (\ref{k1}) on
$[s,+\infty[$. This is a consequence of (\ref{n4}) and (\ref{n5}).
Now we want to give another representation of $(\bar v(\cdot),\bar
w(\cdot))$, by mean of a {\em variation of constant formula}
involving the evolution system generated by the non-autonomous
equation
\begin{equation}\label{na1}
\partial_t(v(t),w(t))=\mathbf
B_{\epsilon}(v(t),w(t))+(0,(1/\epsilon)(-\theta +\widehat{
\partial_u f}(\bar u(t)))\cdot v(t))
\end{equation}
in the space $Z_{-1}$, where $\theta$ is a sufficiently large
positive number to be determined.

\begin{Def}\label{def2} Let $X$ be a Banach space and let $J\subset \R$ be an interval.
A two parameter family of bounded linear operators $U(t,s)$, $s$, $t\in J$, $s\leq t$, is called
an {\em evolution system} on $X$ iff the following conditions are satisfied:
\begin{enumerate}
\item $U(s,s)=I$, $U(t,r)U(r,s)=U(t,s)$ for $s$, $r$, $t\in J$, $s\leq r\leq t$;
\item $(t,s)\mapsto U(t,s)$ is strongly continuous into $\mathcal L(X)$ for $s$, $t\in J$, $s\leq t$.
\end{enumerate}
\end{Def}

We recall the fundamental theorem of Kato (see \cite{K}), which provides sufficient conditions for the existence of an evolution system.
Let $X$ be a Banach space. We denote by $\mathcal G(X)$ the set of all infinitesimal generetors of $C^0$-semigroups of linear operators on $X$.

\begin{Def} \label{def3} Let $X$ be a Banach space and let $J\subset \R$ be an interval.
A one parameter family of  linear operators $A(t)\in \mathcal G(X)$, $t\in J$, is called
 {\em stable} iff there are constants $M>0$, $\beta\in\R$  (called the {\em constants of stabilty}) such that
\begin{equation*}
\|\prod_{j=1}^k(A(t_j)+\lambda)^{-1}  \|_{\mathcal L(X)}\leq M(\lambda-\beta)^{-k}, \quad \lambda>\beta,
\end{equation*}
for any finite family $(t_j)_{j=1}^k$ of points of $J$, with $t_1\leq t_2\leq\cdots\leq t_k$, $k\in\N$.
\end{Def}

\begin{Theorem}[{\bf Theor. 6.1 in \cite{K}} ]\label{thkato}
Let $X$ and $Y$ be Banach spaces, such that $Y$ is densely and continuously embedded in $X$. Let $A(t)$, $t\in \mathcal G(X)$  be a family of linear operators such that:
\begin{enumerate}
\item $(A(t))_{t\in J}$ is stable with constants $M$ and $\beta$;
\item there is a family $(S(t))_{t\in J}$ of isomorphisms of $Y$ to $X$ such that $S(\cdot)$ is strongly continuously differentiable into $\mathcal L(Y,X)$ and
\begin{equation*} S(t)A(t)S(t)^{-1}=A(t)+B(t), \quad B(t)\in\mathcal L(X),
\end{equation*}
where $B(\cdot)$ is strongly continuous into $\mathcal L(X)$;
\item $Y\subset D(A(t))$, so that $A(t)\in\mathcal L(Y,X)$ for $t\in J$, and the map $t\mapsto A(t)$ is norm continuous into $\mathcal L(Y,X)$.
\end{enumerate}
Under these conditions, there exists a unique evolution system $U(t,s)$ on $X$, defined for $s$, $t\in J$, $s\leq t$, with the following properties:
\begin{enumerate}
\item $\|U(t,s)\|_{\mathcal L(X)}\leq Me^{\beta(t-s)}$;
\item $U(t,s)Y\subset Y$ and  $\|U(t,s)|_Y\|_{\mathcal L(Y)}\leq \tilde Me^{\tilde \beta(t-s)}$ for some constants $\tilde M>0$, $\tilde\beta\in\R$;
\item the map $(s,t)\mapsto U(t,s)|_Y$ is strongly continuous in $\mathcal L(Y)$ for $s$, $t\in J$, $s\leq t$;
\item for each fixed $y\in Y$ and $t\in J$, the mapping $s\mapsto U(t,s)y$ is continuously differentiable in $X$ and $(d/ds)U(t,s)y=-U(t,s)A(s)y$, $s\leq t$;
\item for each fixed $y\in Y$ and $s\in J$, the mapping $t\mapsto U(t,s)y$ is continuously differentiable in $X$ and $(d/dt)U(t,s)y=A(t)U(t,s)y$, $s\leq t$.
\end{enumerate}\qed
\end{Theorem}

In order to exploit Kato's theorem, we need to introduce some
notation.
 For $\kappa\in\R$ and $\theta\geq 0$,
define $\mathbf A_\kappa[\theta]:=\mathbf A_\kappa+\theta\mathbf
I\colon H_{\kappa+2}\to H_{\kappa}$.
For $\epsilon\in]0,1]$, $\kappa\in\R$ and $\theta\geq 0$, define
the linear operator $\mathbf B_{\epsilon,\kappa}[\theta]\colon
Z_{\kappa+1}\to Z_{\kappa}$ by
\begin{equation*}
\mathbf
B_{\epsilon,\kappa}[\theta](u,v):=(v,-(1/\epsilon)(v+\mathbf
A_\kappa[\theta] u)),\quad (u,v)\in Z_{\kappa+1}.
\end{equation*}
It follows that $\mathbf B_{\epsilon,\kappa}[\theta]$ is the
infinitesimal generator of a $C^0$-semigroup $\mathbf
T_{\epsilon,\kappa}[\theta](t)$, $t\in[0,+\infty[$, on
$Z_{\kappa}$.

For $t\in\R$, define the operator $\mathbf
C_{\epsilon,-1}(t)\colon Z_{-1}\to Z_{-1}$ by
\begin{equation*}
\mathbf
C_{\epsilon,-1}(t)(u,v):=(0,(1/\epsilon)\widehat{\partial_u
f}(\bar u(t))\cdot u).
\end{equation*}
Notice that, by (\ref{n5}), the mapping $t\mapsto\mathbf
C_{\epsilon,-1}(t)$ is norm continuous into $\mathcal L(Z_{-1})$.
Moreover, by (\ref{n2}), $\mathbf C_{\epsilon,-1}(t)$ maps $Z_0$
into itself. Setting $\mathbf C_{\epsilon,0}(t):=\mathbf
C_{\epsilon,-1}(t)|_{Z_0}$, we get from (\ref{n3}) that the
mapping $t\mapsto\mathbf C_{\epsilon,0}(t)$ is norm continuous
into $\mathcal L(Z_{0})$.

\begin{Prop}\label{verkato} Let $\theta\geq 0$. Set $X:=Z_{-1}$, $Y:=Z_0$, $A(t):=\mathbf
B_{\epsilon,-1}[\theta]+\mathbf C_{\epsilon,-1}(t)$ and
$S(t):=(\mathbf B_{\epsilon,-1}[\theta])^{-1}$, $t\in\R$. Then the
assumptions of Theorem \ref{thkato} are satisfied.
\end{Prop}
\begin{proof} The stability of the family $A(t)$ follows from Prop. 3.5 in \cite{K}. The norm continuity of the mapping $t\mapsto A(t)$
is a consequence of (\ref{n5}). In order to conclude, we shall compute explicitly $S(t)A(t)S(t)^{-1}$. We have that $S(t)A(t)S(t)^{-1}=A(t)+B(t)$, where
\begin{equation}\label{trip1}
B(t)= -\mathbf C_{\epsilon,-1}(t)+\mathbf
B_{\epsilon,-1}[\theta]\mathbf C_{\epsilon,-1}(t)(\mathbf B_{\epsilon,-1}[\theta])^{-1}.
\end{equation}
The first addendum in (\ref{trip1}) is strongly continuous into $\mathcal L(X)$ by (\ref{n5}).  Concerning the second summand, an explicit computation shows that
\begin{equation*}
(\mathbf B_{\epsilon,-1}[\theta])^{-1}(u,v)=(-(\mathbf A_{-1}[\theta])^{-1}(\epsilon v+u),u).
\end{equation*}
It follows that
\begin{multline*}
\mathbf B_{\epsilon,-1}[\theta]\mathbf C_{\epsilon,-1}(t)(\mathbf B_{\epsilon,-1}[\theta])^{-1}(u,v)\\
=(-(1/\epsilon)\widehat{\partial_u f}(\bar u(t))\cdot(\mathbf A_{-1}[\theta])^{-1}(\epsilon v+u),
(1/\epsilon^2)\widehat{\partial_u f}(\bar u(t))\cdot(\mathbf A_{-1}[\theta])^{-1}(\epsilon v+u)).
\end{multline*}
Now it follows from (\ref{n2}) and (\ref{n3}) that the second addendum in (\ref{trip1}) is strongly continuous into $\mathcal L(X)$.
\end{proof}

We denote by $\mathbf U_{\epsilon,-1}[\theta](t,s)$ the evolution
family generated by $\mathbf B_{\epsilon,-1}[\theta]+\mathbf
C_{\epsilon,-1}(t)$ in $Z_{-1}$ and by $\mathbf
U_{\epsilon,0}[\theta](t,s)$ its restriction to $Z_0$. We need the
following lemma.
\begin{Lemma}\label{bridge}
Let $h(\cdot)\colon \R\to L^2(\Omega)$ be a continuous function.
Let $(v_s,w_s)\in Z_{-1}$ and let $(\tilde v(\cdot),\tilde
w(\cdot))\colon[0,+\infty[\to Z_{-1}$ be the unique solution of
\begin{multline}\label{brid1}
(v(t),w(t))=\mathbf T_{\epsilon,-1}(t-s)( v_s, w_s)\\+\int_s^t
\mathbf T_{\epsilon,-1}(t-p)( (0,- (\theta/\epsilon) v(p))+\mathbf
C_{\epsilon,-1} (p)(v(p),w(p))+(0,h(p)))  \,(dp\mid Z_{-1}).
\end{multline}
Then
\begin{equation}\label{brid2}
(\tilde v(t),\tilde w(t))=\mathbf
U_{\epsilon,-1}[\theta](t,s)(v_s,w_s)+\int_s^t \mathbf
U_{\epsilon,-1}[\theta](t,p)(0,h(p)) \,(dp\mid Z_{-1}).
\end{equation}
\end{Lemma}
\begin{proof}
We suppose first that $(v_s,w_s)\in Z_0$. Define
\begin{equation*}
(\check v(t),\check w(t))=\mathbf
U_{\epsilon,-1}[\theta](t,s)(v_s,w_s)+\int_s^t \mathbf
U_{\epsilon,-1}[\theta](t,p)(0,h(p)) \,(dp\mid Z_{-1}).
\end{equation*}
By Theorem 7.1 in \cite{K}, we have that $(\check v(\cdot),\check
w(\cdot))$ is continuously differentiable into $Z_{-1}$,
continuous into $Z_0$, and satisfies
\begin{multline*}
(\partial_t\mid Z_{-1})(\check v(t),\check w(t))=(\mathbf
B_{\epsilon,-1}[\theta]+\mathbf C_{\epsilon,-1}(t))(\check
v(t),\check w(t))+(0,h(t))\\ =\mathbf B_{\epsilon,-1}(\check
v(t),\check w(t))-(0,(\theta/\epsilon)\check v(t)) +\mathbf
C_{\epsilon,-1}(t)(\check v(t),\check w(t))+(0,h(t)).
\end{multline*}
Since the mapping $t\mapsto -(0,(\theta/\epsilon)\check v(t))
+\mathbf C_{\epsilon,-1}(t)(\check v(t),\check w(t))+(0,h(t))$ is
continuous into $Z_0=D(\mathbf B_{\epsilon,-1})$, it follows from
Corollary IV.2.2 in \cite{Pazy} that
\begin{multline*}
(\check v(t),\check w(t))=\mathbf T_{\epsilon,-1}(t-s)( v_s,
w_s)\\+\int_s^t \mathbf T_{\epsilon,-1}(t-p)( (0,-
(\theta/\epsilon) \check v(p))+\mathbf C_{\epsilon,-1} (p)(\check
v(p),\check w(p))+(0,h(p)))  \,(dp\mid Z_{-1}).
\end{multline*}
By the uniqueness of the solution of (\ref{brid1}), we obtain that
$(\check v(\cdot),\check w(\cdot))=(\tilde v(\cdot),\tilde
w(\cdot))$. Finally, if $(v_s,w_s)\in Z_{-1}$, the conclusion
follows from a density argument.
\end{proof}
Now (\ref{k1}) and Lemma \ref{bridge} with $h(t)=(\theta/\epsilon)\bar v(t)$, $t\in\R$, imply that
\begin{equation}\label{k2}
(\bar v(t),\bar w(t))=\mathbf U_{\epsilon,-1}[\theta](t,s)(\bar
v(s),\bar w(s))+\int_s^t \mathbf
U_{\epsilon,-1}[\theta](t,p)(0,(\theta/\epsilon)\bar v(p)) \,(dp\mid Z_{-1}).
\end{equation}
The next step consists in finding suitable decay estimates for  $\mathbf U_{\epsilon,-1}[\theta](t,s)$ and
$\mathbf U_{\epsilon,0}[\theta](t,s)$. To this end, we need to introduce some more notation.

For $\theta\geq 0$ and $\tau\in\R$, we denote by $\mathbf
T_{\epsilon,0}[\theta,\tau](t)$ (resp. by $\mathbf
T_{\epsilon,-1}[\theta,\tau](t)$) the semigroup generated by
$\mathbf B_{\epsilon,0}[\theta]+\mathbf C_{\epsilon,0}(\tau)$ in
$Z_{0}$ (resp. by $\mathbf B_{\epsilon,-1}[\theta]+\mathbf
C_{\epsilon,-1}(\tau)$ in $Z_{-1}$).

For $\theta\geq 0$, we define the following scalar product
in $H^1_0(\Omega)$:
\begin{multline}
\langle u,v\rangle_{H^1_0[\theta]}:=\int_\Omega A(x)\nabla
u(x)\cdot\nabla v(x)\,dx\\+\int_\Omega\beta(x)u(x)v(x)\,dx
+\int_\Omega\theta u(x)v(x)\,dx,\quad u,v\in H^1(\Omega).
\end{multline}
We denote by $\|\cdot\|_{H^1_0[\theta]}$ the corresponding norm.
Moreover, we denote by $\langle\cdot
,\cdot\rangle_{H^{-1}[\theta]}$ the scalar product in
$H^{-1}(\Omega)$ dual to $\langle
\cdot,\cdot\rangle_{H^1_0[\theta]}$, and by
$\|\cdot\|_{H^{-1}[\theta]}$ the corresponding norm.
We have the following estimates:
\begin{equation}\label{equ1}
(\frac{\lambda_1}{\theta+\lambda_1})^{1/2}\|\cdot\|_{H^1_0[\theta]}\leq\|\cdot\|_{H^1_0}\leq
\|\cdot\|_{H^1_0[\theta]}
\end{equation}
and
\begin{equation}\label{equ2}
(\frac{\lambda_1}{\theta+\lambda_1})^{1/2}\|\cdot\|_{H^{-1}}\leq\|\cdot\|_{H^{-1}[\theta]}\leq\|\cdot\|_{H^{-1}}.
\end{equation}
Moreover,
\begin{equation}\label{eigen1}
\|u\|^2_{H^{1}[\theta]}\geq(\lambda_1+\theta) \|u\|^2_{L^2},\quad
u\in H^1_0(\Omega)
\end{equation}
and
\begin{equation}\label{eigen2}
\|u\|^2_{L^2}\geq(\lambda_1+\theta) \|u\|^2_{H^{-1}[\theta]},\quad
u\in L^2(\Omega).
\end{equation}
Notice also that
\begin{equation*}
\langle u,v\rangle_{H^1_0[\theta]}=\langle\mathbf A_0[\theta]
u,v\rangle_{L^2},\quad u\in D(\mathbf A_0[\theta]),\,v\in
H^1_0(\Omega)
\end{equation*}
and
\begin{equation*}
\langle u,v\rangle_{L^2}=\langle\mathbf A_{-1}[\theta]
u,v\rangle_{H^{-1}[\theta]},\quad u\in D(\mathbf A_{-1}[\theta]),\,v\in
L^2(\Omega).
\end{equation*}
For $\theta\geq 0$, we define the following norms in $Z_0$ and
$Z_{-1}$ respectively:
\begin{equation*}
\|(u,v)\|_{Z_{\epsilon,0}[\theta]}:=\|u\|_{H^1_0[\theta]}+\epsilon^{1/2}\|v\|_{L^2},
\quad\quad
\|(u,v)\|_{Z_{\epsilon,-1}[\theta]}:=\|u\|_{L^2}+\epsilon^{1/2}\|v\|_{H^{-1}[\theta]}.
\end{equation*}

For $\theta\geq 0$ and $\tau\in\R$, we define also the following
bilinear form in $H^1_0(\Omega)$:
\begin{multline}
\langle u,v\rangle_{H^1_0[\theta,\tau]}:=\int_\Omega A(x)\nabla
u(x)\cdot\nabla v(x)\,dx+\int_\Omega\beta(x)u(x)v(x)\,dx\\
+\int_\Omega\theta u(x) v(x)\,dx -\int_\Omega \widehat{\partial_u
f}(\bar u(\tau))(x)u(x)v(x)\,dx,\quad u,v\in H^1(\Omega).
\end{multline}
We shall see in a moment that, for sufficiently large $\theta$,
$\langle \cdot,\cdot\rangle_{H^1_0[\theta,\tau]}$
is in fact a scalar product.
We denote by $\|\cdot\|_{H^1_0[\theta,\tau]}$ the corresponding
norm. Moreover, we denote by $\langle\cdot
,\cdot\rangle_{H^{-1}[\theta,\tau]}$ the scalar product in
$H^{-1}(\Omega)$ dual to $\langle
\cdot,\cdot\rangle_{H^1_0[\theta,\tau]}$, and by
$\|\cdot\|_{H^{-1}[\theta,\tau]}$ the corresponding norm. For
$\kappa=0,-1$, define $\mathbf A_{\kappa}[\theta,\tau]:=\mathbf
A_{\kappa}[\theta]-\widehat{\partial_u f}(\bar u(\tau))$ and
notice that
\begin{equation*}
\langle u,v\rangle_{H^1_0[\theta,\tau]}=\langle\mathbf
A_0[\theta,\tau] u,v\rangle_{L^2},\quad u\in D(\mathbf
A_0[\theta,\tau]),\,v\in H^1_0(\Omega)
\end{equation*}
and
\begin{equation*}
\langle u,v\rangle_{L^2}=\langle\mathbf A_{-1}[\theta,\tau]
u,v\rangle_{H^{-1}[\theta,\tau]},\quad u\in D(\mathbf
A_{-1}[\theta,\tau]),\,v\in L^2(\Omega).
\end{equation*}
We need the following lemma.
\begin{Lemma}\label{equiv}
For every $\rho$, with $0<\rho<1$, there exists $\theta_\rho\geq
0$ such that, for all $\theta\geq\theta_\rho$ and $\tau\in\R$,
\begin{equation}\label{equ3}
(1-\rho)^{1/2}\|\cdot\|_{H^1_0[\theta]}\leq\|\cdot\|_{H^1_0[\theta,\tau]}\leq(1+\rho)^{1/2}\|\cdot\|_{H^1_0[\theta]}
\end{equation}
and
\begin{equation}\label{equ4}
(1-\rho)^{1/2}\|\cdot\|_{H^{-1}[\theta,\tau]}\leq\|\cdot\|_{H^{-1}[\theta]}\leq(1+\rho)^{1/2}\|\cdot\|_{H^{-1}[\theta,\tau]}.
\end{equation}
The constant $\theta_\rho$, besides $\rho$, depends only on $R$
and on the constants in Hypotheses \ref{hyp1} and \ref{hyp2}.
\end{Lemma}
\begin{proof}
First we observe that for every $\nu>0$ there is a $C_\nu>0$ with
$$\int_\Omega|\widehat{\partial_u f}(\bar
u(\tau))(x)||u(x)|^2\,dx\leq \nu\int_\Omega|\nabla
u(x)|^2\,dx+C_\nu \int_\Omega|u(x)|^2\,dx $$ for all $u\in
H^1_0(\Omega)$. The constant $C_\nu$, besides $\nu$, depends on
the constants in Hypotheses \ref{hyp1} and \ref{hyp2}. It follows
that
\begin{multline*}
\|u\|_{H^1_0[\theta,\tau]}^2\leq\|u\|_{H^1_0[\theta]}^2+\nu\|u\|_{H^1_0}^2+C_\nu\|u\|_{L^2}^2
\leq\|u\|_{H^1_0[\theta]}^2+\nu\|u\|_{H^1_0[\theta]}^2+\frac{C_\nu}{\theta+\lambda_1}\|u\|_{H^1_0[\theta]}^2.
\end{multline*}
On the other hand,
\begin{equation*}
\|u\|_{H^1_0[\theta,\tau]}^2\geq\|u\|_{H^1_0[\theta]}^2-\nu\|u\|_{H^1_0}^2-C_\nu\|u\|_{L^2}^2
\geq\|u\|_{H^1_0[\theta]}^2-\nu\|u\|_{H^1_0[\theta]}^2-\frac{C_\nu}{\theta+\lambda_1}\|u\|_{H^1_0[\theta]}^2.
\end{equation*}
Choosing first $\nu=\rho/2$ and then $\theta_\rho$ such that
${C_\nu}/(\theta_\rho+\lambda_1)\leq\rho/2$, we obtain
(\ref{equ3}). Estimates (\ref{equ4}) follow from (\ref{equ3}) and
a duality argument.
\end{proof}

Now we are ready to state and prove the desired decay estimates for  $\mathbf U_{\epsilon,-1}[\theta](t,s)$ and
$\mathbf U_{\epsilon,0}[\theta](t,s)$.

\begin{Prop}\label{decay}
There exist $\bar\theta_0>0$, $\delta>0$ $M>0$ such that, for all
$\theta\geq\bar\theta_\rho$,
\begin{equation}\label{dec1}
\|\mathbf U_{\epsilon,-1}[\theta](t,s)\|_{\mathcal
L(Z_{\epsilon,-1}[\theta])}\leq M e^{-\delta(t-s)},\quad t\geq s,
\end{equation}
and
\begin{equation}\label{dec2}
\|\mathbf U_{\epsilon,0}[\theta](t,s)\|_{\mathcal
L(Z_{\epsilon,0}[\theta])}\leq M e^{-\delta(t-s)},\quad t\geq s.
\end{equation}
The constant $\delta$ depends only on the constants in Hypothesis
\ref{hyp1}. The constants $\theta_0$ and $M$ depend on $R$, on the
modulus of continuity $\omega$ of $\bar u(\cdot)$ and on the
constants in Hypotheses \ref{hyp1} and \ref{hyp2}.
\end{Prop}
\begin{proof}
The proof is similar to, and inspired by, the proof of inequality
(3.43) in \cite{HR2}. We begin by proving (\ref{dec1}). Let
$\theta\geq\theta_{1/2}$, where $\theta_{1/2}$ is given by Lemma \ref{equiv}, let $s\in\R$ and let $(v_s,w_s)\in Z_{-1}$. Set
$(v(t),w(t)):=\mathbf U_{\epsilon,-1}[\theta](t,s)(v_s,w_s)$,
$t\geq s$. The same argument exploited in the proof of Lemma
\ref{bridge} shows that, for $s\leq \tau\leq t$,
\begin{multline}\label{dec3}
(v(t),w(t))=\mathbf T_{\epsilon,-1}[\theta,\tau](t-\tau)(v(\tau),w(\tau))\\
+\int_\tau^t\mathbf T_{\epsilon,-1}[\theta,\tau](t-p)
(\mathbf C_{\epsilon,-1}(p)-\mathbf C_{\epsilon,-1}(\tau))(v(p),w(p))\,(dp\mid Z_{-1}).
\end{multline}
Let $\delta$ be a positive constant, with
$\delta\leq\min\{1/2,\lambda_1/2\}$, and let $\eta$  be a positive
constant to be fixed later. For $j\in\N_0$, we define the
intervals $I_j:=[s+j\eta,s+(j+1)\eta]$. For $j\in\N_0$, we
introduce the following family of energy functionals on $Z_{-1}$:
\begin{equation}
E_{\theta,j}(v,w):=\frac12\epsilon\|\delta
v+w\|_{H^{-1}[\theta,s+j\eta]}^2+\frac12\|v\|_{L^2}^2+
\frac12(\epsilon\delta^2-\delta)\|v\|_{H^{-1}[\theta,s+j\eta]}^2.
\end{equation}
Moreover, we define
\begin{equation}
E_{\theta}(v,w):=\frac12\epsilon\|\delta
v+w\|_{H^{-1}[\theta]}^2+\frac12\|v\|_{L^2}^2+
\frac12(\epsilon\delta^2-\delta)\|v\|_{H^{-1}[\theta]}^2.
\end{equation}
Since $\epsilon\in]0,1]$ and $\delta\leq\min\{1/2,\lambda_1/2\}$,
a direct computation using (\ref{eigen2}) shows that, for all
$\theta\geq\theta_{1/2}$,
\begin{equation}\label{dec4}
\frac14\|(v,w)\|_{Z_{\epsilon,-1}[\theta]}\leq E_{\theta}(v,w)\leq
\frac34 \|(v,w)\|_{Z_{\epsilon,-1}[\theta]},\quad (v,w)\in Z_{-1}.
\end{equation}
Moreover, by Lemma \ref{equiv}, for every $\rho$, with $0<\rho\leq 1/2$,
there exists $\theta_\rho\geq\theta_{1/2}$ such that, for all
$\theta\geq\theta_\rho$ and all $j\in\N_0$,
\begin{equation}\label{serve1}
(1-\rho)E_{\theta,j}(v,w)\leq E_{\theta}(v,w)\leq
(1+\rho)E_{\theta,j}(v,w), \quad (v,w)\in Z_{-1}.
\end{equation}
 An
elementary, but quite tedious computation, using (\ref{dec3} )
with $\tau=s+j\eta$ and Theorem 2.6 in \cite{PR2}, shows that the
mapping $t\mapsto E_{\theta,j}(v(t),w(t))$ is differentiable on
$I_j$, and
\begin{multline}
\frac{d}{dt}E_{\theta,j}(v(t),w(t))+2\delta
E_{\theta,j}(v(t),w(t))= (2\epsilon\delta-1)\|\delta
v(t)+w(t)\|_{H^{-1}(\theta,s+j\eta)}^2\\ +\langle\delta v(t)+w(t),
(\widehat{\partial_uf}(\bar u(t))-\widehat{\partial_uf}(\bar
u(s+j\eta)))v(t) \rangle_{H^{-1}(\theta,s+j\eta)}.
\end{multline}
Take $\rho$, with $0<\rho\leq1/2$, and take $\theta\geq\theta_\rho$.
Using Cauchy-Schwartz inequality, inequalities (\ref{equ2}),
(\ref{equ4}) and (\ref{n5}), and the fact that
$(2\epsilon\delta-1)\leq-1/2$, we obtain
\begin{multline}
\frac{d}{dt}E_{\theta,j}(v(t),w(t))+2\delta
E_{\theta,j}(v(t),w(t))\\ \leq\frac12
\|(\widehat{\partial_uf}(\bar u(t))-\widehat{\partial_uf}(\bar
u(s+j\eta)))v(t)\|^2_{H^{-1}(\theta,s+j\eta)}\\ \leq
\frac12(1-\rho)^{-1} \|(\widehat{\partial_uf}(\bar
u(t))-\widehat{\partial_uf}(\bar u(s+j\eta)))v(t)\|^2_{H^{-1}}\\
\leq \frac12(1-\rho)^{-1} \|\widehat{\partial_uf}(\bar
u(t))-\widehat{\partial_uf}(\bar u(s+j\eta))\|^2_{\mathcal L(L^2,
H^{-1})}\|v(t)\|_{L^2}^2\\\leq
 \frac12(1-\rho)^{-1}\tilde C(1+2R^\alpha)^2\|\bar u(t)-\bar
 u(s+j\eta)\|^{2\beta}_{H^1_0}\|v(t)\|_{L^2}^2\\
 \leq \frac12(1-\rho)^{-1}\tilde
 C(1+2R^\alpha)^2\omega(\eta)^{2\beta}\|v(t)\|_{L^2}^2\\
 \leq (1-\rho)^{-1}\tilde
 C(1+2R^\alpha)^2\omega(\eta)^{2\beta}E_{\theta,j}(v(t),w(t)).
\end{multline}
Now, recalling  that $\rho\leq(1/2)$, we choose $\eta$ in such a way
that $$2\tilde
 C(1+2R^\alpha)^2\omega(\eta)^{2\beta}\leq\delta.$$
With this choice, we have
\begin{equation}
\frac{d}{dt}E_{\theta,j}(v(t),w(t))+\delta
E_{\theta,j}(v(t),w(t))\leq0.
\end{equation}
It follows that, for $t\in I_j$,
\begin{equation}\label{won1}
E_{\theta,j}(v(t),w(t))\leq
e^{-\delta(t-(s+j\eta))}E_{\theta,j}(v(s+j\eta),w(s+j\eta)).
\end{equation}
Iterating inequality (\ref{won1}) and taking into account
(\ref{serve1}), we obtain, for $j\in\N_0$ and  $t\in I_j$, that
\begin{equation}\label{won2}
E_{\theta}(v(t),w(t))\leq(\frac{1+\rho}{1-\rho})^{j+1}
e^{-\delta(t-s)}E_{\theta}(v_s,w_s).
\end{equation}
We are still free to choose $\rho\in]0,1/2]$. At this point, we
observe that $t-s\geq j\eta$. Therefore, we choose $\rho$ in such
a way that $$(\frac{1+\rho}{1-\rho}) e^{-\delta \eta/2} \leq 1.$$
With this choice, we obtain that, for $\theta\geq\theta_\rho$,
\begin{equation}\label{won3}
E_{\theta}(v(t),w(t))\leq(\frac{1+\rho}{1-\rho})
e^{-\delta/2(t-s)}E_{\theta}(v_s,w_s),\quad t\geq s.
\end{equation}
Finally, putting together (\ref{dec4}) and (\ref{won3}), we obtain
(\ref{dec1}).

In order to prove (\ref{dec1}) we proceed in the same way. Let
$\theta\geq\theta_{1/2}$, let $s\in\R$ and let $(v_s,w_s)\in Z_{0}$. Set
$(v(t),w(t)):=\mathbf U_{\epsilon,0}[\theta](t,s)(v_s,w_s)$,
$t\geq s$. Again, for $s\leq \tau\leq t$, we have
\begin{multline}
(v(t),w(t))=\mathbf
T_{\epsilon,-1}[\theta,\tau](t-\tau)(v(\tau),w(\tau))\\
+\int_\tau^t\mathbf T_{\epsilon,-1}[\theta,\tau](t-p) (\mathbf
C_{\epsilon,-1}(p)-\mathbf
C_{\epsilon,-1}(\tau))(v(p),w(p))\,(dp\mid Z_{-1}).
\end{multline}
Since $(v(\cdot),w(\cdot))$ is continuous into $Z_0$ by Theorem
\ref{thkato}, we have that
\begin{multline}\label{deca3}
(v(t),w(t))=\mathbf
T_{\epsilon,0}[\theta,\tau](t-\tau)(v(\tau),w(\tau))\\
+\int_\tau^t\mathbf T_{\epsilon,0}[\theta,\tau](t-p) (\mathbf
C_{\epsilon,0}(p)-\mathbf
C_{\epsilon,0}(\tau))(v(p),w(p))\,(dp\mid Z_{0}).
\end{multline}
Let $\delta$ be as above and let $\eta$  be a positive constant to
be fixed later. For $j\in\N_0$, we define the intervals $I_j$ as
above. For $j\in\N_0$, we introduce the following family of energy
functionals on $Z_{0}$:
\begin{equation}
\tilde E_{\theta,j}(v,w):=\frac12\epsilon\|\delta
v+w\|_{L^2}^2+\frac12\|v\|_{H^1_0[\theta,s+j\eta]}^2+
\frac12(\epsilon\delta^2-\delta)\|v\|_{L^2}^2.
\end{equation}
Moreover, we define
\begin{equation}
\tilde E_{\theta}(v,w):=\frac12\epsilon\|\delta
v+w\|_{L^2}^2+\frac12\|v\|_{H^1_0[\theta]}^2+
\frac12(\epsilon\delta^2-\delta)\|v\|_{L^2}^2.
\end{equation}
Since $\epsilon\in]0,1]$ and $\delta\leq\min\{1/2,\lambda_1/2\}$,
a direct computation using (\ref{eigen1}) shows that, for all
$\theta\geq\theta_{1/2}$,
\begin{equation}\label{deca4}
\frac14\|(v,w)\|_{Z_{\epsilon,0}[\theta]}\leq\tilde
E_{\theta}(v,w)\leq \frac34
\|(v,w)\|_{Z_{\epsilon,0}[\theta]},\quad (v,w)\in Z_{0}.
\end{equation}
Moreover, by Lemma \ref{equiv}, for every $\rho$, with $0<\rho\leq 1/2$,
there exists $\theta_\rho>0$ such that, for all
$\theta\geq\theta_\rho$ and all $j\in\N_0$,
\begin{equation}\label{serve2}
(1-\rho)\tilde E_{\theta}(v,w)\leq \tilde E_{\theta,j}(v,w)\leq
(1+\rho)\tilde E_{\theta}(v,w), \quad (v,w)\in Z_{0}.
\end{equation}
Again, using (\ref{deca3} ) with $\tau=s+j\eta$ and Theorem 2.6 in
\cite{PR2}, we see that that the mapping $t\mapsto \tilde
E_{\theta,j}(v(t),w(t))$ is differentiable on $I_j$, and
\begin{multline}
\frac{d}{dt}\tilde E_{\theta,j}(v(t),w(t))+2\delta \tilde
E_{\theta,j}(v(t),w(t))= (2\epsilon\delta-1)\|\delta
v(t)+w(t)\|_{L^2}^2\\ +\langle\delta v(t)+w(t),
(\widehat{\partial_uf}(\bar u(t))-\widehat{\partial_uf}(\bar
u(s+j\eta)))v(t) \rangle_{L^2}.
\end{multline}
Again, take $\rho$, with $0<\rho\leq1/2$, and take
$\theta\geq\theta_\rho$. Using Cauchy-Schwartz inequality,
inequalities (\ref{equ1}), (\ref{equ3}) and (\ref{n3}), and the
fact that $(2\epsilon\delta-1)\leq-1/2$, we obtain
\begin{multline}
\frac{d}{dt}\delta E_{\theta,j}(v(t),w(t))+2\delta \tilde
E_{\theta,j}(v(t),w(t))\\ \leq\frac12
\|(\widehat{\partial_uf}(\bar u(t))-\widehat{\partial_uf}(\bar
u(s+j\eta)))v(t)\|^2_{L^2}\\ \leq \frac12
\|\widehat{\partial_uf}(\bar u(t))-\widehat{\partial_uf}(\bar
u(s+j\eta))\|^2_{\mathcal L(H^1_0, L^2)}\|v(t)\|_{H^1_0}^2\\ \leq
\frac12 (1-\rho)^{-1}\|\widehat{\partial_uf}(\bar
u(t))-\widehat{\partial_uf}(\bar u(s+j\eta))\|^2_{\mathcal
L(H^1_0, L^2)}\|v(t)\|_{H^1_0[\theta,s+j\eta]}^2\\ \leq
 \frac12(1-\rho)^{-1}\tilde C(1+2R^\alpha)^2\|\bar u(t)-\bar
 u(s+j\eta)\|^{2\beta}_{H^1_0}\|v(t)\|_{H^1_0[\theta,s+j\eta]}^2\\
 \leq \frac12(1-\rho)^{-1}\tilde
 C(1+2R^\alpha)^2\omega(\eta)^{2\beta}\|v(t)\|_{H^1_0[\theta,s+j\eta]}^2\\
 \leq (1-\rho)^{-1}\tilde
 C(1+2R^\alpha)^2\omega(\eta)^{2\beta}\tilde E_{\theta,j}(v(t),w(t)).
\end{multline}
Now we proceed exactly as in the final part of the proof of
(\ref{dec1}): recalling that $\rho\leq(1/2)$,  we choose
$\eta$ in such a way that $$2\tilde
 C(1+2R^\alpha)^2\omega(\eta)^{2\beta}\leq\delta.$$
With this choice, we obtain that, for $t\in I_j$,
\begin{equation}\label{wond1}
\tilde E_{\theta,j}(v(t),w(t))\leq e^{-\delta(t-(s+j\eta))}\tilde
E_{\theta,j}(v(s+j\eta),w(s+j\eta)).
\end{equation}
Iterating inequality (\ref{wond1}) and taking into account
(\ref{serve2}), we obtain, for $j\in\N_0$ and  $t\in I_j$, that
\begin{equation}\label{wond2}
\tilde E_{\theta}(v(t),w(t))\leq(\frac{1+\rho}{1-\rho})^{j+1}
e^{-\delta(t-s)}\tilde E_{\theta}(v_s,w_s).
\end{equation}
At this point, we choose $\rho$ in such a way that
$$(\frac{1+\rho}{1-\rho}) e^{-\delta \eta/2} \leq 1.$$ With this
choice, we obtain that, for $\theta\geq\theta_\rho$,
\begin{equation}\label{wond3}
\tilde E_{\theta}(v(t),w(t))\leq(\frac{1+\rho}{1-\rho})
e^{-\delta/2(t-s)}\tilde E_{\theta}(v_s,w_s),\quad t\geq s.
\end{equation}
Finally, putting together (\ref{deca4}) and (\ref{wond3}), we
obtain (\ref{dec2}).
\end{proof}

Now we can conclude the proof of Theorem \ref{th1}. Fix
$\theta\geq\theta_0$, where $\theta_0$ is given by Proposition
\ref{decay}. Thanks to the decay estimate (\ref{dec1}), we can let
$s$ tend to $-\infty$ in (\ref{k2}), so as to obtain
\begin{equation}\label{gulp1}
(\bar v(t),\bar w(t))=\int_{-\infty}^t \mathbf
U_{\epsilon,-1}[\theta](t,p)(0,(\theta/\epsilon)\bar v(p))
\,(dp\mid Z_{-1})
\end{equation}
for all $t\in\R$. Now observe that the mapping
$p\mapsto(0,(\theta/\epsilon)\bar v(p))$ is continuous into $Z_0$.
Therefore, thanks to the decay estimate (\ref{dec2}), we deduce
that
\begin{equation}\label{gulp2}
(\bar v(t),\bar w(t))=\int_{-\infty}^t \mathbf
U_{\epsilon,0}[\theta](t,p)(0,(\theta/\epsilon)\bar v(p))
\,(dp\mid Z_{0}).
\end{equation}
It follows that $(\bar v(\cdot),\bar w(\cdot))$ is continuous into
$Z_0$ and, for all $t\in\R$,
\begin{multline*}
\|(\bar v(t),\bar w(t))\|_{Z_{\epsilon,0}[\theta]}\leq
\int_{-\infty}^t Me^{-\delta(t-p)}\|(0,(\theta/\epsilon)\bar
v(p))\|_{Z_{\epsilon,0}[\theta]} \,dp\\ \leq \int_{-\infty}^t
Me^{-\delta(t-p)}\epsilon^{1/2}\|(\theta/\epsilon)\bar
v(p))\|_{L^2} \,dp\leq\int_{-\infty}^t
Me^{-\delta(t-p)}(\theta/\epsilon)R
\,dp=\frac{MR\theta}{\delta\epsilon}.
\end{multline*}
It follows that $(\bar u(\cdot),\bar v(\cdot))$ is continuously
differentiable into $Z_0$, with
\begin{equation*}\begin{cases}
(\partial_t\mid H_1)\bar u(t)=(\partial_t\mid H_0)\bar u(t)=\bar
v(t)\\ \epsilon(\partial_t\mid H_{0})\bar v(t)=\epsilon\bar w(t)
=\epsilon(\partial_t\mid H_{-1})\bar v(t)=-\bar v(t)-\mathbf
A_{-1}\bar u(t)+\hat f(\bar u(t)).
\end{cases}
\end{equation*}
Now we have that
\begin{equation}\label{urka}
\mathbf A_{-1}\bar u(t)=- \epsilon\bar w(t)-\bar v(t)+\hat f(\bar
u(t)).
\end{equation}
The right hand side of (\ref{urka}) is a continuous function of
$t$ into $L^2(\Omega)$. Then $\bar u(\cdot)$ is a continuous
function into $D(\mathbf A_0)$, and
\begin{equation*}
\|\mathbf A_0\bar u(t)\|_{L^2}\leq\epsilon\|\bar
w(t)\|_{L^2}+\|\bar v(t)\|_{L^2}+\|\hat f(\bar u(t))\|_{L^2}
\end{equation*}
Summing up, we obtain that
\begin{equation*}
\sup_{t\in\R}(\|\mathbf A_0\bar u(t)\|_{L^2}^2+\|\bar
v(t)\|_{H^1_0}^2+\epsilon\|(\partial_t\mid H_0)\bar
v(t)\|_{L^2}^2)\leq 4\frac{MR\theta}{\delta\epsilon}+\tilde
C^2(1+R^3)^2.
\end{equation*}
This concludes the proof of Theorem \ref{th1}.

\section{Proof of Theorem 2}
Throughout this section, for every $\epsilon\in]0,1]$, we denote
by $(\bar u_\epsilon(\cdot), \bar v_\epsilon(\cdot))\colon\R\to
Z_0$ a fixed bounded full solution of (\ref{eq2}), such that
$\sup_{t\in\R}(\|\bar u_\epsilon(t)\|_{H^1_0}^2+\epsilon\|\bar
v_\epsilon(t)\|_{L^2}^2)\leq R$. It follows from Theorem \ref{th1} that
$(\bar u_\epsilon(\cdot),\bar v_\epsilon(\cdot))$ is continuous
into $Z_1$ and continuously differentiable into $Z_0$, with
\begin{equation}\label{new1}\begin{cases}
(\partial_t\mid H_1)\bar u_\epsilon(t)=\bar v_\epsilon(t)\\
\epsilon(\partial_t\mid H_{0})\bar v_\epsilon(t)=-\bar
v_\epsilon(t)-\mathbf A_{0}\bar u_\epsilon(t)+\hat f(\bar
u_\epsilon(t)).
\end{cases}
\end{equation}
Moreover, for every $\epsilon\in]0,1]$ there exists a positive
constant $\tilde R_\epsilon$ such that
\begin{equation}\label{new2}
\sup_{t\in\R}(\|\mathbf A_0 \bar u_\epsilon(t)\|_{L^2}^2+\|\bar
v_\epsilon(t)\|_{H^1_0}^2+\epsilon\|(\partial_t\mid H_0)\bar
v_\epsilon(t)\|_{L^2}^2)\leq \tilde R_\epsilon.
\end{equation}
Set $\bar w(t):=(\partial_t\mid H_{0})\bar v(t)$, $t\in\R$. Using
(\ref{n2}) we see that $(\bar v_\epsilon(\cdot),\bar
w_\epsilon(\cdot))$ is  continuously differentiable into $Z_{-1}$,
and
\begin{equation*}\begin{cases}
(\partial_t\mid H_{0})\bar v_\epsilon(t)=\bar w_\epsilon(t)\\
\epsilon(\partial_t\mid H_{-1})\bar w_\epsilon(t)=-\bar
w_\epsilon(t)-\mathbf A_{-1}\bar v_\epsilon(t)+\widehat{\partial_u
f}(\bar u_\epsilon(t))\cdot\bar v_\epsilon(t)
\end{cases}
\end{equation*}
Let $\theta\geq0$. Since the mapping $t\mapsto(0,
\widehat{\partial_u f}(\bar u_\epsilon(t))\cdot\bar v_\epsilon(t))$
is continuous into $Z_{0}=D(\mathbf B_{\epsilon,-1}[\theta])$, it
follows from Theorem II.1.3 in \cite{Gold} that, for $s$,
$t\in\R$, with $s\leq t$, $(\bar v_\epsilon(\cdot),\bar
w_\epsilon(\cdot))$ satisfies the equality
\begin{multline*}
(\bar v_\epsilon(t),\bar w_\epsilon(t))=\mathbf
T_{\epsilon,-1}[\theta](t-s)(\bar v_\epsilon (s),\bar w_\epsilon
(s))\\+\int_s^t \mathbf
T_{\epsilon,-1}[\theta](t-p)(0,(1/\epsilon)\widehat{\partial_u
f}(\bar u_\epsilon(p))\cdot\bar
v_\epsilon(p)+(\theta/\epsilon)\bar v_\epsilon(p))\,(dp\mid
Z_{-1}).
\end{multline*}
Finally, since $(\bar v_\epsilon(\cdot),\bar w_\epsilon(\cdot))$
is continuous into $Z_{0}$, it follows  that $(\bar
v_\epsilon(\cdot),\bar w_\epsilon(\cdot))$ satisfies the equality
\begin{multline}\label{new3}
(\bar v_\epsilon(t),\bar w_\epsilon(t))=\mathbf
T_{\epsilon,0}[\theta](t-s)(\bar v_\epsilon (s),\bar w_\epsilon
(s))\\+\int_s^t \mathbf
T_{\epsilon,0}[\theta](t-p)(0,(1/\epsilon)\widehat{\partial_u
f}(\bar u_\epsilon(p))\cdot\bar
v_\epsilon(p)+(\theta/\epsilon)\bar v_\epsilon(p))\,(dp\mid
Z_{0}).
\end{multline}
Let $\delta$ be a positive constant, with
$\delta\leq\min\{1/2,\lambda_1/2\}$. We define the following
energy functional on $Z_0$:
\begin{equation}
\tilde E_{\epsilon,\theta}(v,w):=\frac12\epsilon\|\delta
v+w\|_{L^2}^2+\frac12\|v\|_{H^1_0[\theta]}^2+
\frac12(\epsilon\delta^2-\delta)\|v\|_{L^2}^2.
\end{equation}
A direct computation using (\ref{eigen1}) shows that, for all
$\theta\geq0$,
\begin{equation}\label{new4}
\frac14\|(v,w)\|_{Z_{\epsilon,0}[\theta]}\leq\tilde
E_{\epsilon,\theta}(v,w)\leq \frac34
\|(v,w)\|_{Z_{\epsilon,0}[\theta]},\quad (v,w)\in Z_{0}.
\end{equation}
Moreover, by Lemma \ref{equiv}, for every $\rho$, with $0<\rho<1$,
there exists $\theta_\rho>0$ such that, for all
$\theta\geq\theta_\rho$, all $t\in\R$ and all $(v,w)\in Z_{0}$,
\begin{equation}\label{new5}
(1-\rho)\tilde E_{\epsilon,\theta}(v,w)\leq \tilde
E_{\epsilon,\theta}(v,w)+\frac12\int_\Omega \widehat{ \partial_u
f}(\bar u_\epsilon(t))(x)|v(x)|^2\,dx\leq (1+\rho)\tilde
E_{\epsilon,\theta}(v,w).
\end{equation}
Fixing $\rho=1/2$ and setting $\theta_*:=\theta_{1/2}$, we obtain
from (\ref{new4}) and (\ref{new5}) that, for all
$\theta\geq\theta_*$, all $t\in\R$ and all $(v,w)\in Z_{0}$,
\begin{equation}\label{new6}
\frac18\|(v,w)\|_{Z_{\epsilon,0}[\theta]}\leq \tilde
E_{\epsilon,\theta}(v,w)+\frac12\int_\Omega \widehat{
\partial_u f}(\bar u_\epsilon(t))(x)|v(x)|^2\,dx\leq
\frac98\|(v,w)\|_{Z_{\epsilon,0}[\theta]}.
\end{equation}

We define the following function:
\begin{equation}
\Lambda_{\epsilon,\theta}(t):=E_{\epsilon,\theta}(\bar
v_\epsilon(t),\bar w_\epsilon(t))+\frac12\int_\Omega \widehat{
\partial_u f}(\bar u_\epsilon(t))(x)|\bar v_\epsilon(t)(x)|^2\,dx
\end{equation}
We need the following lemma, whose proof is left to the reader:
\begin{Lemma}\label{diffc2}
Assume Hypothesis (\ref{hyp3}). Define the mapping
\begin{equation*}
\mathcal G_\epsilon(t):=\frac12\int_\Omega \widehat{ \partial_u
f}(\bar u_\epsilon(t))(x)|\bar v_\epsilon(t)(x)|^2\,dx.
\end{equation*}
Then $\mathcal G_\epsilon(\cdot)$ is continuously differentiable,
and
\begin{multline*}
\frac{d}{dt}\mathcal G_\epsilon(t)=\frac12\int_\Omega \widehat{
\partial_{uu} f}(\bar u_\epsilon(t))(x)\bar v_\epsilon(t)(x)|\bar v_\epsilon(t)(x)|^2\,dx\\+
\int_\Omega \widehat{ \partial_u f}(\bar u_\epsilon(t))(x)\bar
v_\epsilon(t)(x)\bar w_\epsilon(t)(x) \,dx.
\end{multline*}
\qed\end{Lemma}

Using (\ref{new3}), Theorem 2.6 in \cite{PR2} and Lemma
\ref{diffc2}, we see that $\Lambda_{\epsilon,\theta}(\cdot)$ is
differentiable and
\begin{multline}\label{new7}
\frac{d}{dt}\Lambda_{\epsilon,\theta}(t)+2\delta\Lambda_{\epsilon,\theta}(t)
=(2\delta\epsilon-1)\|\bar w_\epsilon(t)+\delta\bar
v_\epsilon(t)\|_{L^2}^2\\+\langle \bar w_\epsilon(t)+\delta\bar
v_\epsilon(t),\theta\bar v_\epsilon(t)\rangle_{L^2}
+\frac12\int_\Omega \widehat{
\partial_{uu} f}(\bar u_\epsilon(t))(x)\bar v_\epsilon(t)(x)|\bar
v_\epsilon(t)(x)|^2\,dx.
\end{multline}
By Hypothesis \ref{hyp3}, we have:
\begin{multline*}
\frac12\int_\Omega \widehat{
\partial_{uu} f}(\bar u_\epsilon(t))(x)\bar v_\epsilon(t)(x)|\bar
v_\epsilon(t)(x)|^2\,dx\\ \leq\frac12\int_\Omega C_1(1+|\bar
u_\epsilon(t)(x)|)|\bar v_\epsilon(t)(x)||\bar
v_\epsilon(t)(x)|^2\,dx\\ \leq \frac12 C_1\|\bar
v_\epsilon(t)\|_{L^2}\|\bar v_\epsilon(t)\|_{L^4}^2+\frac12
C_1\|\bar u_\epsilon(t)\|_{L^6}\|\bar v_\epsilon(t)\|_{L^2}\|\bar
v_\epsilon(t)\|_{L^6}^2\\ \leq \frac12C_2(1+R)\|\bar
v_\epsilon(t)\|_{L^2}\|\bar v_\epsilon(t)\|_{H^1_0}^2.
\end{multline*}
It follows that, for every $\nu>0$, there exists $C_\nu>0$ such
that
\begin{multline*}
\frac12\int_\Omega \widehat{
\partial_{uu} f}(\bar u_\epsilon(t))(x)\bar v_\epsilon(t)(x)|\bar
v_\epsilon(t)(x)|^2\,dx \leq\nu\|\bar
v_\epsilon(t)\|_{H^1_0}^2+C_\nu \|\bar
v_\epsilon(t)\|_{L^2}^2\|\bar v_\epsilon(t)\|_{H^1_0}^2 \\\leq
\nu\|\bar v_\epsilon(t)\|_{H^1_0[\theta]}^2+C_\nu \|\bar
v_\epsilon(t)\|_{L^2}^2\|\bar v_\epsilon(t)\|_{H^1_0[\theta]}^2.
\end{multline*}
Then, choosing $\nu\leq\delta$ and using Cauchy-Schwartz
inequality in (\ref{new7}), we get
\begin{equation}\label{new8}
\frac{d}{dt}\Lambda_{\epsilon,\theta}(t)+\delta\Lambda_{\epsilon,\theta}(t)
\leq (\theta^2/2)\|\bar v_\epsilon(t)\|^2_{L^2}+2C_\nu \|\bar
v_\epsilon(t)\|_{L^2}^2\Lambda_{\epsilon,\theta}(t).
\end{equation}
We need the following lemma.
\begin{Lemma}\label{integral}
There exists a positive constant $K$ such that, for all
$\epsilon\in]0,1]$,
\begin{equation*}
\int_{-\infty}^{+\infty}\|\bar v_\epsilon(t)\|_{L^2}^2\,dt\leq K.
\end{equation*}
The constant $K$ depends only on $R$ and on the constants in
Hypotheses \ref{hyp1} and \ref{hyp3}. In particular, $K$ is
independent of $\epsilon$.
\end{Lemma}
\begin{proof}
Define the standard Lyapunov functional
\begin{equation*}
L(u,v):=\epsilon\frac12\|v\|^2_{L^2}+\frac12\|u\|_{H^1_0}^2-\int_\Omega
F(x,u(x))\,dx,\quad (u,v)\in H^1_0(\Omega)\times L^2(\Omega),
\end{equation*}
where $F(x,u):=\int_0^uf(x,s)ds$. Then the mapping $t\mapsto
L(\bar u(t),\bar v(t))$ is differentiable and
\begin{equation*}
\frac{d}{dt}L(\bar u(t),\bar v(t))=-\|\bar v(t)\|^2_{L^2}
\end{equation*}
(for details, see the proof of Proposition 4.1 in \cite{PR2}).
Then, for every $t_1<t_2$,
\begin{equation*}
\int_{t_1}^{t_2}\|\bar v_\epsilon(t)\|_{L^2}^2\,dt\leq |L(\bar
u(t_1),\bar v(t_1) )|+|L(\bar u(t_1),\bar v(t_2))|\leq K(R),
\end{equation*}
where $K(R)$ is a suitable constant depending on $R$ and on the
constants of Hypothesis \ref{hyp3}.
\end{proof}
Let $\sigma$, $t$, $\tau\in\R$, with $\sigma\geq t\geq\tau$. We
multiply (\ref{new8}) by $e^{\delta( t-\tau)-\int_\tau^t2C_\nu
\|\bar v_\epsilon(s)\|_{L^2}^2\,ds}$ and we obtain that
\begin{equation*}
\frac{d}{dt}(e^{\delta( t-\tau)-\int_\tau^t2C_\nu \|\bar
v_\epsilon(s)\|_{L^2}^2\,ds}\Lambda_{\epsilon,\theta}(t))\leq(\theta^2/2)
e^{\delta( t-\tau)-\int_\tau^t2C_\nu \|\bar
v_\epsilon(s)\|_{L^2}^2\,ds}\|\bar v_\epsilon(t)\|^2_{L^2}.
\end{equation*}
Integrating on $[\tau,\sigma]$, we get
\begin{multline*}
e^{\delta( \sigma-\tau)-\int_\tau^\sigma2C_\nu \|\bar
v_\epsilon(s)\|_{L^2}^2\,ds}\Lambda_{\epsilon,\theta}(\sigma)\\\leq\Lambda_{\epsilon,\theta}(\tau)
+(\theta^2/2)\int_\tau^\sigma e^{\delta( t-\tau)-\int_\tau^t2C_\nu
\|\bar v_\epsilon(s)\|_{L^2}^2\,ds}\|\bar
v_\epsilon(t)\|^2_{L^2}\,dt.
\end{multline*}
It follows from Lemma \ref{integral} that
\begin{multline*}
\Lambda_{\epsilon,\theta}(\sigma)\leq e^{-\delta(
\sigma-\tau)+\int_\tau^\sigma2C_\nu \|\bar
v_\epsilon(s)\|_{L^2}^2\,ds}\Lambda_{\epsilon,\theta}(\tau)\\
+(\theta^2/2)\int_\tau^\sigma e^{-\delta( \sigma-t)+\int_t^\sigma
2C_\nu \|\bar v_\epsilon(s)\|_{L^2}^2\,ds}\|\bar
v_\epsilon(t)\|^2_{L^2}\,dt\\ \leq e^{-\delta( \sigma-\tau)}
e^{2C_\nu K}\Lambda_{\epsilon,\theta}(\tau)+(\theta^2/2)Ke^{2C_\nu
K}.
\end{multline*}
Using (\ref{new6}),(\ref{new2}) and (\ref{equ1}) we get
\begin{multline*}
(1/8)\|(\bar v_\epsilon(\sigma),\bar
w_\epsilon(\sigma))\|^2_{Z_{\epsilon,0}[\theta]}\\\leq
(9/8)e^{-\delta( \sigma-\tau)} e^{2C_\nu K}\|(\bar
v_\epsilon(\tau),\bar
w_\epsilon(\tau))\|^2_{Z_{\epsilon,0}[\theta]}+(\theta^2/2)Ke^{2C_\nu
K}\\ \leq (9/8)(\frac{\theta+\lambda_1}{\lambda_1})e^{-\delta(
\sigma-\tau)} e^{2C_\nu K}\tilde R_\epsilon+(\theta^2/2)Ke^{2C_\nu
K}.
\end{multline*}
Letting $\tau$ tend to $-\infty$, we finally get
\begin{equation}
\|(\bar v_\epsilon(\sigma),\bar
w_\epsilon(\sigma))\|^2_{Z_{\epsilon,0}}\leq 4\theta^2Ke^{2C_\nu
K},\quad\sigma\in\R.
\end{equation}
This last inequality, together with (\ref{new1}), yields
\begin{equation*}
\sup_{t\in\R}(\|\mathbf A_0 \bar u_\epsilon(t)\|_{L^2}^2+\|\bar
v_\epsilon(t)\|_{H^1_0}^2+\epsilon\|(\partial_t\mid H_0)\bar
v_\epsilon(t)\|_{L^2}^2)\leq \tilde R,
\end{equation*}
where $\tilde R$ depends only on the constants in Hypotheses
\ref{hyp1} and \ref{hyp3} and on $R$. This concludes the proof of
Theorem \ref{th2}.

\section{An application: upper semicontinuity of attractors}
In this section we assume Hypotheses \ref{hyp1} and \ref{hyp3}. Moreover, we make  the following structure assumption on the nonlinearity $f(x,u)$.

\begin{Hyp}\label{hyp4}There exists a positive number $\mu$ and a function $c(\cdot)\in L^2(\Omega)$ such that:
\begin{enumerate}
\item $f(x,u)u-\mu F(x,u)\leq c(x)$;
\item $F(x,u)\leq c(x)$.
\end{enumerate}
Here, $F(x,u):=\int_0^uf(x,s)\,ds$, $(x,u)\in\Omega\times\R$.
\end{Hyp}
It was proved in \cite{PR2} that under Hypotheses \ref{hyp1}, \ref{hyp3} and \ref{hyp4}, for every $\epsilon\in]0,1]$, equation (\ref{eq1}) (more precisely: its mild formulation (\ref{eq2})) generates a global semiflow in $H^1_0(\Omega)\times L^2(\Omega)$,
possessing a compact global attractor $\mathcal A_\epsilon$. Moreover, there exists a positive constant $R$ such that
\begin{equation*}
\sup_{\epsilon\in]0,1]}\sup\{\|u\|_{H^1_0}^2+\epsilon\|v\|_{L^2}^2\mid(u,v)\in\mathcal A_\epsilon\}\leq R.
\end{equation*}
Consider now the formal limit of (\ref{eq1}) as $\epsilon\to 0$, i.e. the parabolic equation
\begin{equation}\label{peq1}
\begin{aligned}
u_t+\beta(x)u-\sum_{ij}(a_{ij}(x)u_{x_j})_{x_i}&=f(x,u),&&(t,x)\in[0,+\infty[\times\Omega,\\
u&=0,&&(t,x)\in[0,+\infty[\times\partial\Omega
\end{aligned}\end{equation}
with Cauchy datum $u(0)=u_0$. Again we rewrite (\ref{peq1}) as an integral evolution equation in the space $H^1_0(\Omega)$, namely
\begin{equation}\label{peq2}
u(t)=e^{-\mathbf A_0t}u_0+\int_0^t e^{-\mathbf A_0(t-s)}\hat f(u(s))\,ds,
\end{equation}
where $e^{-\mathbf A_0 t}$, $t\geq 0$, is the analytic semigroup generated by the positive  selfadjoint operator $\mathbf A_0$ in $L^2(\Omega)$.
It was proved in \cite{PR1} that under Hypotheses \ref{hyp1}, \ref{hyp3} and \ref{hyp4},  equation (\ref{peq1}) (more precisely: its mild formulation (\ref{peq2})) generates a global semiflow in $H^1_0(\Omega)$,
possessing a compact global attractor $\tilde{\mathcal A}$. Moreover, $\tilde{\mathcal A}\subset D(\mathbf A_0)$ and $\tilde{\mathcal A}$ is compact in $ D(\mathbf A_0)$.

Let $\Gamma\colon D(\mathbf A_0)\to H^1_0(\Omega)\times L^2(\Omega)$ be defined by $\Gamma(u):=(u,\mathbf A_0 u+\hat f(u))$. Set $\mathcal A_0:=\Gamma(\tilde{\mathcal A})$.
In \cite{PR3} the following result was proved.

\begin{Theorem}[{\bf Theor. 1.4 in \cite{PR3}}]
The family $(\mathcal A_\epsilon)_{\epsilon\in[0,1]}$ is upper semicontinuous at $\epsilon=0$ with respect to the topology of $H^1_0(\Omega)\times H^{-1}(\Omega)$, i.e.
\begin{equation*}
\lim_{\epsilon\to 0}\sup_{y\in\mathcal A_\epsilon}\inf_{z\in\mathcal A_0}\|y-z\|_{H^1_0\times H^{-1}}=0.
\end{equation*}\qed
\end{Theorem}

This result is not completely satisfactory. The optimal result would be to obtain upper semicontinuity with respect to the
topology of $H^1_0(\Omega)\times L^2(\Omega)$. Actually, thanks to Theorem \ref{th2}, we are now able to prove the optimal result.

\begin{Theorem}\label{newtheorem} The family $(\mathcal A_\epsilon)_{\epsilon\in[0,1]}$ is upper semicontinuous at $\epsilon=0$ with respect to the topology of $H^1_0(\Omega)\times L^2(\Omega)$, i.e.
\begin{equation*}
\lim_{\epsilon\to 0}\sup_{y\in\mathcal A_\epsilon}\inf_{z\in\mathcal A_0}\|y-z\|_{H^1_0\times L^2}=0.
\end{equation*}
\end{Theorem}

Indeed, the main ingredient in the proof of Theor. 4.1 in \cite{PR3} is the following Theorem.

\begin{Theorem}[{\bf Theor. 3.8 in \cite{PR3}}]\label{subsequence}
Let $(\epsilon_n)_n$ be a sequence of positive numbers converging to $0$. For each $n\in\N$ let $z_n=(u_n,v_n)\colon \R\to H^1_0(\Omega)\times L^2(\Omega)$ be a bounded full solution of (\ref{eq2})
such that
$$\sup_{n\in\N}\sup_{t\in\R}( \|u_n(t)\|_{H^1_0}^2+\epsilon_n\|v_n(t)\|_{L^2}^2)\leq R<\infty.$$
Then a subsequence of $(z_n)_n$ converges in $H^1_0(\Omega)\times H^{-1}(\Omega)$,
uniformly on compact subsets of $\R$, to a function $z\colon\R\to H^1_0(\Omega)\times L^2(\Omega)$ with $z=(u,v)$, where $u$ is a solution of (\ref{peq2}) and $v=(\partial_t\mid L^2(\Omega))u$.\qed
\end{Theorem}

If in Theorem \ref{subsequence} we assume also that, for each $n\in\N$, the function $u_n(\cdot)$ is uniformly continuous, then it follows
from Theorem \ref{th2} that the sequence $(v_n(\cdot))_n$ is bounded in $L^\infty(\R,H^1_0(\Omega))$. Interpolation between $H^1_0(\Omega)$ and $H^{-1}(\Omega)$  then implies that
$v_n(t)\to v(t)$ in $L^2(\Omega)$ uniformly for $t$ lying in compact subsets of $\R$. Now using Lemma \ref{modcont} and an obvious contradiction argument one easily completes the proof of Theorem \ref{newtheorem}.

\end{document}